\documentclass{article}

\usepackage[pdftex]{graphicx,xcolor,hyperref}
\usepackage{float}
\usepackage{amsmath,amssymb,bm}
\usepackage{amsthm}
\usepackage{ascmac}
\usepackage{array}
\usepackage{mathtools}
\mathtoolsset{showonlyrefs=true}

\newtheorem{theorem}{Theorem}[section]
\newtheorem{lemma}[theorem]{Lemma}
\newtheorem{corollary}[theorem]{Corollary}
\newtheorem{proposition}[theorem]{Proposition}

\numberwithin{equation}{section}

\theoremstyle{definition}
\newtheorem{definition}[theorem]{Definition}
\newtheorem{example}[theorem]{Example}
\newtheorem{remark}[theorem]{Remark}

\newcommand{\veps}{\varepsilon}

\begin{document}

\title{Asymptotic behaviors of the integrated density of states for random Schr\"{o}dinger operators associated with Gibbs Point Processes}
\author{Yuta Nakagawa\footnote{Graduate School of Human and Environmental Studies, Kyoto University, Japan.}
\footnote{E-mail: nakagawa.yuta.58n@st.kyoto-u.ac.jp}}

\maketitle

\begin{abstract}
The asymptotic behaviors of the integrated density of states \(N(\lambda)\) of Schr\"{o}dinger operators with nonpositive potentials associated with Gibbs point processes are studied. It is shown that for some Gibbs point processes, the leading terms of \(N(\lambda)\) as \(\lambda\downarrow-\infty\) coincide with that for a Poisson point process, which is known. Moreover, for some Gibbs point processes corresponding to pairwise interactions, the leading terms of \(N(\lambda)\) as \(\lambda\downarrow-\infty\) are determined, which are different from that for a Poisson point process.
\flushleft{{\bf MSC2020 subject classifications:} Primary 82B44, Secondary 60G55; 60G60; 60K35.}
\end{abstract}

\section{Introduction}
We consider the random Schr\"{o}dinger operator on \(L^2(\mathbb{R}^d,dx)\) defined by
\begin{equation}\label{S operator}
H_{\eta_{\omega}}=-\Delta+V_{\eta_{\omega}},\quad V_{\eta_{\omega}}(x)= \int_{\mathbb{R}^d} u_{0}(x-y)\eta_{\omega}(dy),
\end{equation}
where \(u_{0}\) is a nonpositive measurable function on \(\mathbb{R}^d\), and \(\eta_{\omega}\) is a point process (see Section \ref{Gibbs sec}). 
We call \(u_0\) the \emph{single site potential}. 
We assume that \(\eta_{\omega}\) is stationary and ergodic (see Section \ref{Gibbs sec}). The \emph{integrated density of states (IDS)} \(N(\lambda)\) of the Schr\"{o}dinger operator is the nondecreasing function formally given by
\begin{equation}
\lim_{L\to\infty} \frac{1}{|\Lambda_{L}|}\#\{\text{eigenvalues of \(H_{\eta_{\omega},L}^{D}\) less than or equal to \(\lambda\)}\},
\end{equation}
where \(\Lambda_L\) is the box \((-L/2,L/2)^d\subset\mathbb{R}^d\), \(|\Lambda_{L}|\) denotes Lebesgue measure of \(\Lambda_{L}\), and \(H_{\eta_{\omega},L}^{D}\) is the operator \(H_{\eta_{\omega}}\) restricted to \(\Lambda_{L}\) with Dirichlet boundary condition.
Because of the ergodicity of \(\eta_{\omega}\), both \(N(\lambda)\) and the spectrum of \(H_{\eta_{\omega}}\) are independent of \(\omega\) almost surely. Moreover, \(N(\lambda)\) increases only on the spectrum.
See \cite{CL, PF} for the precise definition and the properties of the IDS. 

The asymptotic behaviors of the IDS near the infimums of the spectra are well-studied. 
For a stationary Poisson point process (see \cite{Dereudre} for the definition) and a nonpositive and integrable single site potential \(u_0\) which is continuous at the origin and has a minimum value of less than zero there, the spectrum is \(\mathbb{R}\) apart from some exceptional cases (see \cite{AIKN} for details), and it holds that
\begin{equation}\label{Pastur tails}
\log N(\lambda)\sim -\frac{\lambda\log|\lambda|}{u_0(0)}\quad(\lambda\downarrow-\infty),
\end{equation}
which is proved by Pastur in \cite{Pastur} (see also \cite[(9.4) Theorem]{PF}), where we write \(f(\lambda)\sim g(\lambda)\ (\lambda\downarrow-\infty)\) if \(f(\lambda)/g(\lambda)\) converges to one as \(\lambda\downarrow-\infty\).
See \cite{KP} for the asymptotic behavior for a singular nonpositive single site potential, and \cite{DoV, Nakao, Pastur} for that for a nonnegative single site potential.

The aim of this work is to investigate the asymptotic behaviors of \(N(\lambda)\) as \(\lambda\downarrow-\infty\) for nonpositive single site potentials and Gibbs point processes: point processes with interactions between the points.
We mainly deal with pairwise interaction processes (i.e. Gibbs point processes corresponding to pairwise interactions: the energy of the points \(\{x_j\}\) is \(\sum_{i<j}\varphi(x_i-x_j)\), where \(\varphi\) is a nonnegative symmetric function on \(\mathbb{R}^d\) with compact support). In \cite{Sznitman}, Sznitman proved a result that is equivalent to determining the leading term of the asymptotic behavior of the IDS for a nonnegative single site potential with compact support and a pairwise interaction process.

From the proof of \eqref{Pastur tails} in \cite[(9.4) Theorem]{PF}, we can find that for a Poisson point process, the probability that many points exist in a small domain affects the leading term of the IDS.
Therefore, for a Gibbs point process corresponding to a weak interaction that does not prevent the points from gathering (e.g. Example \ref{weak pair}), we expect that the IDS satisfies \eqref{Pastur tails}. This is proved in Theorem \ref{thm1}. However, for a general Gibbs point process, \eqref{Pastur tails} does not hold. In Corollary \ref{Strauss}, for a pairwise interaction process satisfying some conditions, we determine the leading term of the asymptotic behavior of the IDS:
\begin{equation}\label{main cor}
\log N(\lambda) \sim -\frac{\varphi(0)}{2\|u_0\|^2_{S}}\lambda^2\quad (\lambda\downarrow-\infty),
\end{equation}
where \(\|u_0\|_{S}^2\), defined in Section \ref{Pairwise sec}, depends only on \(u_0\) and the support \(S\) of \(\varphi\). 
This implies that the IDS decays much faster than that for a Poisson point process.
In this case, because of the repulsion of the points, the probability that some clusters of many points occur affects the leading term of the IDS (see \eqref{lower prob} and \eqref{pairwise leading}). 
The main tools for the proof of \eqref{main cor} are Proposition \ref{prop}, which is used in the proof of \eqref{Pastur tails} in \cite[(9.4) Theorem]{PF}, and the upper estimate of the Laplace functional in Proposition \ref{upper lap}.

This paper is organized as follows. Section \ref{Gibbs sec} introduces the Gibbs point processes (see e.g. \cite{Dereudre, Preston, Ruelle}). In Section \ref{Weak sec}, we prove that the leading term of the asymptotic behavior of the IDS for a Gibbs point process corresponding to a weak interaction is identical to \eqref{Pastur tails}. In Section \ref{Pairwise sec}, we treat a pairwise interaction process satisfying some conditions and determine the leading term of the asymptotic behavior of the corresponding IDS, which is different from \eqref{Pastur tails}.

\section{Gibbs point processes}\label{Gibbs sec}

Let \((\mathcal{C},\mathcal{F})\) denote the space of all locally finite measures on \((\mathbb{R}^d,\mathcal{B}(\mathbb{R}^d))\) which can be written as a countable sum \(\sum_{j=1}^{n}\delta_{x_j}\ (n\in\mathbb{Z}_{\geq0}\cup\{+\infty\},\ x_j\in\mathbb{R}^d,\ x_i\neq x_j\ \text{for any}\ i\neq j)\), equipped with the \(\sigma\)-algebra \(\mathcal{F}\) generated by \(\{M_{\Lambda}\}_{\Lambda\in\mathcal{B}(\mathbb{R}^d)}\), where \(\delta_x\) denotes the Dirac measure at \(x\in\mathbb{R}^d\), and for every \(\Lambda\in\mathcal{B}(\mathbb{R}^d)\), \(M_{\Lambda}\) is the function on \(\mathcal{C}\) defined by \(M_{\Lambda}(\eta)=\#(\Lambda\cap\mathrm{supp}\,\eta)\).
We note that all \(\eta\in\mathcal{C}\) is \emph{simple} (i.e. \(\eta(\{x\})\leq 1\) for any \(x\in\mathbb{R}^d\)).
We call a \(\mathcal{C}\)-valued random element a \emph{point process}.

Let \(\mathcal{C}_{f}\) be the set of all finite measures in \(\mathcal{C}\). We introduce an energy function to define the interaction between the points. 

\begin{definition}
An \emph{energy function} is a measurable function \(U\) from \(\mathcal{C}_{f}\) to \(\mathbb{R}\cup\{+\infty\}\) such that:
\begin{itemize}
\item \(U\) maps the null measure to \(0\);
\item if \(U(\eta)=+\infty\), then \(U(\delta_x+\eta)=+\infty\) for all \(x\in\mathbb{R}^d\setminus\mathrm{supp}\,\eta\);
\item \(U(\tau_{x}\eta)=U(\eta)\) for any \(\eta\in\mathcal{C}_{f}\) and any \(x\in\mathbb{R}^d\), where \(\tau_{x}\eta(\cdot)=\eta(\tau_{-x}\,\cdot)\), and \(\tau_x\) is the translation by vector \(x\).
\end{itemize}
\end{definition}

We introduce a pairwise energy function: an energy function corresponding to a pairwise interaction.
\begin{definition}
An energy function defined by
\begin{equation}\label{pairwise function}
U(\eta)=\sum_{\substack{\{x,y\}\subset\,\mathrm{supp}\,\eta\\ x\neq y}}\varphi(x-y)
\end{equation}
is called a \emph{pairwise energy function}, where \(\varphi\) is a measurable and symmetric (i.e. \(\varphi(x)=\varphi(-x)\)) function from \(\mathbb{R}^d\) to \(\mathbb{R}_{\geq 0}\cup\{+\infty\}\) with compact support.
\end{definition}

For a pairwise energy function \(U\) defined by \eqref{pairwise function} and a bounded \(\Lambda\in\mathcal{B}(\mathbb{R}^d)\), we define
\begin{equation}
U_{\Lambda}(\eta)=\sum_{\substack{\{x,y\}\subset\,\mathrm{supp}\,\eta\\ \{x,y\}\cap\Lambda\neq\emptyset,\ x\neq y}}\varphi(x-y)\quad(\eta\in\mathcal{C}_{f}).
\end{equation}
For an energy function \(U\) except for pairwise energy functions and a bounded \(\Lambda\in\mathcal{B}(\mathbb{R}^d)\), we set
\begin{equation}
U_{\Lambda}(\eta)=U(\eta)-U(\eta_{\Lambda^c})\quad(\eta\in\mathcal{C}_{f}),
\end{equation}
where \(\eta_{\Lambda}\) denotes the measure \(\eta(\cdot\cap\Lambda)\), \(\Lambda^c=\mathbb{R}^d\setminus \Lambda\), and \(\infty-\infty=0\) for convenience.
We can see \(U_{\Lambda}(\eta)\) means the variation of the energy when adding \(\eta_{\Lambda}\) to \(\eta_{\Lambda^c}\). 

In this paper, we assume the finite range property:
\begin{definition}
We say that an energy function \(U\) has a \emph{finite range \(R>0\)} if for all bounded \(\Lambda\in\mathcal{B}(\mathbb{R}^d)\) and all \(\eta\in\mathcal{C}_{f}\), it holds that 
\begin{equation}
U_{\Lambda}(\eta)=U_{\Lambda}(\eta_{\Lambda+B(0,R)}),
\end{equation}
where \(B(x,R)\) is the closed ball centered at \(x\) with radius \(R\), and \(\Lambda+B(0,R)\) denotes the set \(\{x+y\mid x\in\Lambda,\ y\in B(0,R)\}\).
\end{definition}
A pairwise energy function defined by \eqref{pairwise function} has a finite range \(R>0\) if and only if the support of \(\varphi\) is included in \(B(0,R)\).

For an energy function \(U\) with a finite range \(R>0\), we can define
\begin{equation}
U_{\Lambda}(\eta)=U_{\Lambda}(\eta_{\Lambda+B(0,R)}),\ h(x,\eta)=U_{\{x\}}(\eta+\delta_{x}) \quad(\eta\in\mathcal{C},\ x\in\mathbb{R}^d\setminus\mathrm{supp}\,\eta).
\end{equation}
The function \(h\) is called the \emph{local energy function}, which means the variation of the energy when adding \(\delta_x\) to \(\eta\).

Let \(\mathcal{F}_{\Lambda}\) denote the \(\sigma\)-algebra generated by \(\{M_{\Lambda'}\}_{\Lambda'\subset\Lambda,\ \Lambda'\in\mathcal{B}(\mathbb{R}^d)}\). When an energy function \(U\) has a local energy function bounded below, for every \(\gamma\in\mathcal{C}\) and every bounded \(\Lambda\in\mathcal{B}(\mathbb{R}^d)\) with positive Lebesgue measure, we define the probability measure \(P_{\Lambda,\gamma}\) on \((\mathcal{C},\mathcal{F}_{\Lambda})\) by
\begin{equation}
P_{\Lambda,\gamma}(d\eta)=\frac{1}{Z_{\Lambda}(\gamma)}e^{-U_{\Lambda}(\eta_{\Lambda}+\gamma_{\Lambda^c})}\pi^1(d\eta),
\end{equation}
where \(\pi^z\) is the distribution of the Poisson point process with intensity \(z\) (see \cite{Dereudre} for the definition), and \(Z_{\Lambda}(\gamma)\) is a normalizing constant:
\begin{equation}
Z_{\Lambda}(\gamma)= \int e^{-U_{\Lambda}(\eta_{\Lambda}+\gamma_{\Lambda^c})}\pi^1(d\eta).
\end{equation} 
We see
\begin{equation}\label{partition lower}
Z_{\Lambda}(\gamma)\geq \int_{\{M_{\Lambda}=0\}}\pi^1(d\eta)=e^{-|\Lambda|}>0.
\end{equation} 
Moreover, if the local energy function \(h\) is bounded below by \(a\in\mathbb{R}\), since for any distinct points \(x_1,\ldots,x_n\in\Lambda\),
\begin{equation}
U_{\Lambda}(\sum_{j=1}^n \delta_{x_j}+\gamma_{\Lambda^c})=\sum_{j=1}^{n} h(x_j,\sum_{i=1}^{j-1}\delta_{x_i}+\gamma_{\Lambda^c})\geq an,
\end{equation}
we have
\begin{equation}\label{partition upper}
Z_{\Lambda}(\gamma)\leq\int_{\mathcal{C}}e^{-aM_{\Lambda}(\eta)}\pi(d\eta)=\exp\big(|\Lambda|(e^{-a}-1)\big)<+\infty.
\end{equation}

Now we introduce the Gibbs point processes.
\begin{definition}\label{Gibbs def}
We assume that an energy function \(U\) has a finite range, and its local energy function is bounded below.
A point process is a \emph{Gibbs point process} for the energy function \(U\) if its distribution \(P\) on \(\mathcal{C}\) satisfies the following conditions:
\begin{itemize}
\item for all bounded \(\Lambda\in\mathcal{B}(\mathbb{R}^d)\) with positive Lebesgue measure and all bounded \(\mathcal{F}\)-measurable function \(f\), it holds that
\begin{equation}\label{DLReq}
\int_{\mathcal{C}}f(\eta)P(d\eta)=\int_{\mathcal{C}}\int_{\mathcal{C}}f(\eta_{\Lambda}+\gamma_{\Lambda^c})P_{\Lambda,\gamma}(d\eta)P(d\gamma);
\end{equation}
\item for \(P\)-almost all \(\gamma\in\mathcal{C}\), \(U(\eta)<+\infty\) whenever \(\eta\in\mathcal{C}_{f}\) and \(\mathrm{supp}\,\eta\subset\mathrm{supp}\,\gamma\).
\end{itemize} 
\end{definition}

The equation \eqref{DLReq} is called the \emph{Dobrushin-Lanford-Ruelle equation (DLR equation)}. 
We call a Gibbs point process for a pairwise energy function a \emph{pairwise interaction process}. 
We note that if \(U\) is identically zero, the corresponding Gibbs point process is the Poisson point process with intensity one.

We say a point process is \emph{stationary} if its distribution \(P\) satisfies \(P(\tau_x A)=P(A)\) for all \(x\in\mathbb{R}^d\) and all \(A\in\mathcal{F}\). Under the assumption in Definition \ref{Gibbs def}, at least one stationary Gibbs point process for \(U\) exists (see \cite{DDG, DeV}). 

A point process with a distribution \(P\) is called \emph{ergodic} if \(P(A)\) is either zero or one for any \(A\in\mathcal{F}\) such that \(P(A\ominus \tau_x A)=0\) for all \(x\in\mathbb{R}^d\), where \(A\ominus B\) denotes the symmetric difference of \(A\) and \(B\). 
If an energy function with a finite range \(R>0\) has a local energy function bounded below by \(a\in\mathbb{R}^d\), and \(R^d e^{-a}\) is sufficiently small, then the corresponding DLR equation has a unique solution (see \cite{HT} for more details). The uniqueness guarantees the stationarity and ergodicity of the corresponding Gibbs point process (see \cite[Section 4]{Preston}).

We end this section by introducing stochastic domination. We say that a measurable function \(f\) on \((\mathcal{C},\mathcal{F})\) is \emph{increasing}, if \(f(\eta)\leq f(\gamma)\) whenever \(\mathrm{supp}\,\eta\subset\mathrm{supp}\,\gamma\). If an energy function \(U\) has a local energy function bounded below by \(-\log z\) for some \(z>0\), then the corresponding probability measure \(P_{\Lambda,\gamma}\) is \emph{stochastically dominated} by \(\pi^z\): for any bounded \(\Lambda\in\mathcal{B}(\mathbb{R}^d)\) and any bounded \(\mathcal{F}_{\Lambda}\)-measurable increasing function \(f\), it holds that
\begin{equation}
\int_{\mathcal{C}}f\,dP_{\Lambda,\gamma} \leq \int_{\mathcal{C}}f\,d\pi^{z},
\end{equation} 
(see \cite{GK} for the proof).
If there is a Gibbs point process for \(U\), then for its distribution \(P\), we also have
\begin{equation}\label{dom eq}
\int_{\mathcal{C}}f\,dP \leq \int_{\mathcal{C}}f\,d\pi^{z},
\end{equation}
from the DLR equation.


\section{Weak interactions}\label{Weak sec}

As we see below, for a bounded integrable single site potential and a stationary and ergodic Gibbs point process with a local energy function bounded below, the Schr\"{o}dinger operator \eqref{S operator} is essentially self-adjoint on \(C_{c}^{\infty}\) almost surely, and the corresponding IDS exists, where \(C_{c}^{\infty}\) denotes the set of all smooth functions on \(\mathbb{R}^d\) with compact support. We note that all Schr\"{o}dinger operators in this paper have these properties. 

Let \(P\) be the distribution of the Gibbs point process. From \eqref{dom eq}, for any \(p\in [1,\infty)\) and any bounded \(\Lambda\in\mathcal{B}(\mathbb{R}^d)\), we have
\begin{equation}\label{self-adjointness}
\int_{\mathcal{C}}\int_{\Lambda}|V_{\eta}(x)|^p \,dx\, P(d\eta) \leq \int_{\mathcal{C}}\int_{\Lambda}|V_{\eta}(x)|^p \,dx\, \pi^{z}(d\eta)<\infty,
\end{equation}
where \(z>0\) is a constant such that the local energy function is bounded below by \(-\log z\).
The essential self-adjointness is proved in \cite[Proposition V.3.2]{CL} using \eqref{self-adjointness} and the stationarity. For the existence of the IDS, see \cite[Chapter VI]{CL}.

Now, in this section, we consider a Gibbs point process for an energy function satisfying the following condition.

\begin{description}
\item[(W)]There exists a positive constant \(\veps_0\) and a positive function \(r\) on \((1,\infty)\), such that
\begin{equation}
\lim_{x\to\infty}r(x)=0,\quad\lim_{x\to\infty}\frac{\log r(x)}{\log x}=0,
\end{equation}
and for any \(0<\veps\leq\veps_0\),
\begin{equation}
\sup_{M_{\mathbb{R}^d}(\eta_{B(0,r(x))})=\lceil(1+\veps)x\rceil} U(\eta_{B(0,r(x))})=o(x\log x)\quad(x\to\infty),
\end{equation}
where \(\lceil x\rceil\) denotes the least integer greater than or equal to \(x\).
\end{description}
Condition \textbf{(W)} implies that the interaction between many points in a small region is sufficiently weak. 

\begin{example}[Energy function with a bounded local energy function]\label{weak area}
Some energy functions have bounded local energy functions. An example is an \emph{area energy function}:
\begin{equation}
U(\sum_{j=1}^{n}\delta_{x_j})=\Big|\bigcup_{j=1}^{n}B(x_j,R)\Big|,
\end{equation}
which has the local energy function \(h\) satisfying \(0\leq h\leq |B(0,R)|\), where \(R>0\) is a constant.
If a local energy function \(h\) of an energy function \(U\) is bounded above by \(a\in\mathbb{R}\), then we have \(U(\eta)\leq aM_{\mathbb{R}^d}(\eta)\) for all \(\eta\in\mathcal{C}_{f}\).
Hence, the energy function \(U\) satisfies condition \textbf{(W)}.  
\end{example}

\begin{example}[Pairwise energy function]\label{weak pair}
In general, the local energy functions of the pairwise energy functions are unbounded. However, a pairwise energy function such that the closer the two points approach, the weaker the interaction gets, could satisfy condition \textbf{(W)}.
For example, a pairwise energy function defined by \eqref{pairwise function} such that \(\varphi(x)=O(\exp(-|x|^{-p}))\ (|x|\to 0)\ \) for some \(p>0\), satisfies condition \textbf{(W)} (let \(r(x)=2^{-1}(\log x)^{-1/p}\)). 
\end{example}

The main result in this section is the following.

\begin{theorem}\label{thm1}
Let an energy function \(U\) with a finite range satisfy condition \textbf{(W)}, its local energy function be bounded below, and a single site potential \(u_0\) be nonpositive, integrable, continuous at the origin, and have a minimum value of less than zero there. We assume that a Gibbs point process for \(U\) is stationary and ergodic.
Then the corresponding IDS satisfies that
\begin{equation}\label{thm1 eq}
\log N(\lambda)\sim -\frac{\lambda\log|\lambda|}{u_{0}(0)}\quad(\lambda\downarrow-\infty).
\end{equation}
\end{theorem}

The asymptotic behavior \eqref{thm1 eq} coincides with \eqref{Pastur tails}: the asymptotic behavior of the IDS for a Poisson point process.

One of the main tools in this paper is the following, which is proved by the path integral representation of the Laplace transform of the IDS and Dirichlet-Neumann bracketing (see \cite[(9.1) Theorem and (9.2) Theorem]{PF} and \cite{CL} for the proof).

\begin{proposition}\label{prop}
For a stationary and ergodic Gibbs point process with a local energy function bounded below and a bounded integrable single site potential \(u_0\), the following hold:

\begin{itemize}
\item for any \(\veps>0\) and any \(\delta\in(0,1/2)\), there exists \(\lambda_0<0\) such that for any \(\lambda<\lambda_0\),
\begin{equation}\label{lower}
\log N(\lambda) \geq \log P\big(\sup_{|x|<\delta}V_{\cdot}(x)\leq (1+\veps)\lambda\big);
\end{equation}
\item for any \(t>1\) and any \(\lambda\in\mathbb{R}\),
\begin{equation}\label{upper}
\log N(\lambda) \leq \lambda t+ \log\big(\int_{\mathcal{C}}\exp(-t\int_{\mathbb{R}^d}u_0(-x)\eta(dx)) P(d\eta)\big),
\end{equation}
\end{itemize}
where \(P\) is the distribution of the Gibbs point process.
\end{proposition}
The integral on the right-hand side of \eqref{upper} is called the \emph{Laplace functional}.

\begin{proof}[Proof of Theorem \ref{thm1}]
Let \(P\) be the distribution of the Gibbs point process.

\emph{(Lower bound)} 
For simplicity, we put \(u(x)=-u_{0}(-x)\) (i.e. \(u\) is nonnegative, and \(V_{\eta}(x)=-\int u(y-x)\eta(dy)\)).
We assume that the energy function \(U\) has a finite range \(R>0\), and the local energy function is bounded below by \(-\log z\) for some \(z>0\). 

Fix nonempty bounded \(\Lambda\in\mathcal{B}(\mathbb{R}^d)\) and \(n\in\mathbb{N}\). Let \(A=\{M_{\Lambda}=n,\ M_{\Lambda_{R}\setminus\Lambda}=0\}\in\mathcal{F}_{\Lambda_R}\), where \(\Lambda_R=\Lambda+B(0,R)\).
For all \(\eta\in A\) and \(P\)-almost all \(\gamma\in\mathcal{C}\), we have
\begin{equation}
\begin{split}
U_{\Lambda_{R}}(\eta_{\Lambda_{R}}+\gamma_{(\Lambda_{R})^c})&=U(\eta_{\Lambda}+\gamma_{\Lambda_{2R}\setminus\Lambda_{R}})-U(\gamma_{\Lambda_{2R}\setminus\Lambda_{R}})\\
&=U(\eta_{\Lambda})\leq \sup_{M_{\mathbb{R}^d}(\eta_{\Lambda})=n}U(\eta_{\Lambda}).
\end{split}
\end{equation}
Hence, from the DLR equation and \eqref{partition upper}, we obtain
\begin{equation}\label{weak low}
\begin{split}
P(M_{\Lambda}=n)&\geq P(A)\\
&\geq \int_{\mathcal{C}}\int_{\mathcal{C}} \frac{1}{Z_{\Lambda_{R}}(\gamma)} 1_{A}(\eta_{\Lambda_{R}})e^{-U_{\Lambda_{R}}(\eta_{\Lambda_{R}}+\gamma_{(\Lambda_{R})^c})}\pi^1(d\eta)P(d\gamma)\\
&\geq e^{-z|\Lambda_{R}|}\frac{|\Lambda|^n}{n!}\exp\big(-\sup_{M_{\mathbb{R}^d}(\eta_{\Lambda})=n}U(\eta_{\Lambda})\big),
\end{split}
\end{equation}
where \(1_{A}\) is the indicator function of \(A\).

Let \(U\) satisfy condition \textbf{(W)} with a constant \(\veps_0>0\) and a function \(r\).
Fix \(\veps\in(0,u(0))\) such that \((1+\veps)/(u(0)-\veps)\leq(1+\veps_0)/u(0)\).
We find \(\delta\in(0,1/2)\) such that \(u(0)-\veps\leq u(x)\) whenever \(|x|\leq 2\delta\), and we have
\begin{equation}
\sup_{|x|<\delta}V_{\eta}(x)\leq -(u(0)-\veps)M_{B(0,\delta)}(\eta).
\end{equation}
From \eqref{lower} and \eqref{weak low}, for all sufficiently small \(\lambda<0\), we get
\begin{equation}
\begin{split}
\log N(\lambda) &\geq \log P\big(-(u(0)-\veps)M_{B(0,\delta)}\leq (1+\veps)\lambda\big)\\
&\geq \log P\big(M_{\Lambda(\lambda)}=n(\lambda) \big)\\
&\geq -\log n(\lambda)! -z|B(0,2R)|\\
&\qquad+n(\lambda)\log|\Lambda(\lambda)|-\sup_{M_{\mathbb{R}^d}(\eta_{\Lambda(\lambda)})=n(\lambda)}U(\eta_{\Lambda(\lambda)}),
\end{split}
\end{equation}
where we put \(\Lambda(\lambda)=B(0,r(|\lambda|/u(0)))\) and \(n(\lambda)=\lceil(1+\veps)|\lambda|/(u(0)-\veps)\rceil\).

From condition \textbf{(W)}, we have
\begin{equation}
n(\lambda)\log|\Lambda(\lambda)|-\sup_{M_{\mathbb{R}^d}(\eta_{\Lambda(\lambda)})=n(\lambda)}U(\eta_{\Lambda(\lambda)})=o(\lambda\log|\lambda|)\quad(\lambda\downarrow-\infty).
\end{equation}
Hence, we obtain
\begin{equation}
\liminf_{\lambda\downarrow-\infty}\frac{\log N(\lambda)}{|\lambda|\log|\lambda|}
\geq\liminf_{\lambda\downarrow-\infty}\frac{-\log n(\lambda)!}{|\lambda|\log|\lambda|}=-\frac{1+\veps}{u(0)-\veps},
\end{equation}
where we use Stirling's formula in the last equation.
This implies that
\begin{equation}
\liminf_{\lambda\downarrow-\infty}\frac{\log N(\lambda)}{|\lambda|\log|\lambda|}\geq\frac{1}{u_{0}(0)}.
\end{equation}

\emph{(Upper bound)} 
From \eqref{dom eq} and \eqref{upper}, we obtain
\begin{equation}
\log N(\lambda)\leq \lambda t + \log\big(\int_{\mathcal{C}}\exp\big(-t\int_{\mathbb{R}^d}u_0(-x)\eta(dx)\big) \pi^z(d\eta)\big).
\end{equation}
We substitute \(t=(\log|\lambda|)/u(0)\), and with the simple calculation of the Laplace functional of the Poisson point process (see \cite[(9.4) Theorem]{PF}), we obtain that
\begin{equation}
\limsup_{\lambda\downarrow-\infty}\frac{\log N(\lambda)}{|\lambda|\log|\lambda|}\leq \frac{1}{u_{0}(0)}.
\end{equation}

\end{proof}


\section{Pairwise interactions}\label{Pairwise sec}
\subsection{Main results}

In this section, we consider a pairwise interaction process corresponding to a pairwise interaction such that the repulsion does not fade as the two points approach each other (cf. Example \ref{weak pair}). For such a pairwise interaction process, the leading term of the asymptotic behavior of the IDS could be affected by the values of the single site potential that are not necessarily the minimum value (cf. \eqref{Pastur tails}). To represent the effect, we define
\begin{equation}
\|u\|_{S}^2=\sup\big\{\sum_{j=1}^{\infty}u(x_j)^2 \mid \{x_j\}_{j=1}^{\infty}\subset\mathbb{R}^d,\ x_i-x_j\in S^c\text{ for any \(i\neq j\)}\big\},
\end{equation}
for every bounded \(S\in\mathcal{B}(\mathbb{R}^d)\) and every measurable function \(u\) on \(\mathbb{R}^d\).
We introduce the following condition associated with the range of the interaction.

\begin{description}
\item[(S)] \(S\in\mathcal{B}(\mathbb{R}^d)\) is nonempty, bounded and symmetric (i.e. \(S=-S\)), and for any \(\alpha\in(0,1)\), 
\begin{equation}
\alpha\overline{S}\subset \mathrm{Int}\,S,
\end{equation}
where \(\mathrm{Int}\,S\) is the interior of \(S\), and for \(a\in\mathbb{R}\), \(aS\) denotes the set \(\{a x\mid x\in S\}\).
\end{description}
We remark that \(S\in\mathcal{B}(\mathbb{R}^d)\) satisfying condition \textbf{(S)} is a bounded star domain (i.e. \(\alpha S\subset S\) for any \(\alpha\in(0,1)\)) including the origin in its interior; however, the converse is false.

The following theorem gives the upper and lower estimates of the asymptotic behavior of the IDS for a pairwise interaction process.

\begin{theorem}\label{strong pair}
Let \(S\in\mathcal{B}(\mathbb{R}^d)\) satisfy condition \textbf{(S)}, and \(u_0\) be a continuous integrable nonpositive single site potential satisfying \(0<\|u_0\|_{S}^{2}<\infty\). Moreover, we assume that a pairwise interaction process for an energy function \(U\) defined by \eqref{pairwise function} having a finite range is stationary and ergodic. Then, we have the following:
\begin{description}
\item[(a)] If \(\varphi\) is upper semi-continuous at the origin, \(\varphi(0)>0\), and \(\varphi(x)=0\) whenever \(x\in S^c\), then it holds that
\begin{equation}\label{strong pair a}
\liminf_{\lambda\downarrow-\infty}\lambda^{-2}\log N(\lambda)\geq -\frac{\varphi(0)}{2\|u_0\|_{S}^2};
\end{equation}
\end{description}
\begin{description}
\item[(b)] If \(u_0\) has a compact support, and there exists a positive constant \(a\) such that \(\varphi(x)\geq a\) whenever \(x\in S\), then it holds that
\begin{equation}\label{strong pair b}
\limsup_{\lambda\downarrow-\infty}\lambda^{-2}\log N(\lambda)\leq -\frac{a}{2\|u_0\|_{S}^2}.
\end{equation}
\end{description}
\end{theorem}

From Theorem \ref{strong pair}, we determine the leading terms of the asymptotic behaviors of the IDS for some pairwise interaction processes:  

\begin{corollary}\label{Strauss}
Let a single site potential \(u_0\) be a nonpositive continuous function that is not identically zero, with compact support, and \(\varphi\) be continuous at the origin and satisfy that \(\varphi(0)>0\), \(\varphi(x)\geq\varphi(0)\) for any \(x\in S\), and \(\varphi(x)=0\) for any \(x\in S^c\), where \(S\in\mathcal{B}(\mathbb{R}^d)\) satisfies condition \textbf{(S)}. We assume that a pairwise interaction process for a pairwise energy function defined by \eqref{pairwise function} is stationary and ergodic.
Then the corresponding IDS satisfies that
\begin{equation}
\log N(\lambda)\sim-\frac{\varphi(0)}{2\|u_0\|_{S}^2}\lambda^2\quad(\lambda\downarrow-\infty).
\end{equation}
\end{corollary}

\begin{remark}
When \(\varphi\) is not upper semi-continuous at the origin and is bounded on a neighborhood there, we can apply Theorem \ref{strong pair} (a) by replacing the value of \(\varphi(0)\) so that \(\varphi\) is upper semi-continuous at the origin. Similarly, Corollary \ref{Strauss} could be applied by replacing \(\varphi(0)\) so that \(\varphi\) is continuous at the origin. These are justified by the fact that the pairwise interaction processes are independent of the value of \(\varphi(0)\). 
\end{remark}

\begin{remark}
Corollary \ref{Strauss} includes the \emph{Strauss processes} (i.e. \(\varphi=a1_{B(0,R)}\) for some \(a,R\in(0,\infty)\)) introduced in \cite{Strauss}. 
In this remark, we fix a single site potential satisfying the condition of Corollary \ref{Strauss}. 
Theorem \ref{strong pair} implies that for a pairwise interaction process such that \(c_1<\varphi<c_2\) on a neighborhood of the origin for some \(c_1,c_2\in(0,\infty)\), the corresponding IDS satisfies that
\begin{equation}\label{loglog}
\log|\log N(\lambda)|\sim 2\log|\lambda|\quad(\lambda\downarrow-\infty).
\end{equation} 
For a \emph{hardcore process} (i.e. \(\varphi=+\infty\) on a neighborhood of the origin), the corresponding IDS does not satisfy \eqref{loglog}. In this case, there is a constant \(\beta\) such that \(V_{\eta}>\beta\) for \(P\)-almost all \(\eta\in\mathcal{C}\), where \(P\) is the distribution of the hardcore process, and hence \(N(\lambda)=0\) whenever \(\lambda\) is less than \(\beta\). 
\end{remark}


\subsection{Proof of Theorem \ref{strong pair} (a)}

To apply \eqref{lower}, we use \eqref{lower prob} unlike either case of a Gibbs point process in Section \ref{Weak sec} or a Poisson point process.
Let \(P\) be the distribution of the pairwise interaction process in Theorem \ref{strong pair} (a).

\begin{proof}[Proof of Theorem \ref{strong pair} (a)]
For simplicity, we put \(u(x)=-u_{0}(-x)\).
We fix \(\veps_0>0\). From the continuity of \(u\) and condition \textbf{(S)}, we can find finite points \(x_1,\ldots x_k\in\mathbb{R}^d\) such that \(u(x_j)\) is positive, \(x_i-x_j\in(\overline{S})^c\ (i\neq j)\), and
\begin{equation}\label{epsilon0}
\|u\|_{S}^2-\veps_0\leq\sum_{j=1}^{k}u(x_j)^2,
\end{equation} 
where \(\overline{S}\) denotes the closure of \(S\).
By continuity of \(u\) and upper semi-continuity at the origin of \(\varphi\), for any \(\veps\in(0,\min_{j} u(x_j))\), there exists \(\delta\in(0,1/2)\) such that:
\begin{itemize}
\item \(x-y\in S^c\) for any \(x\in B(x_i,\delta)\) and any \(y\in B(x_j,\delta)\quad(i\neq j)\);
\item \(u(x)+\veps> u(x_j)\) whenever \(x\in B(x_j,2\delta) \quad(j=1,\ldots,k)\);
\item \(\varphi(x)<\varphi(0)+\veps\) whenever \(x\in B(0,2\delta)\). 
\end{itemize}

We put \(\Lambda_j=B(x_j,\delta)\), \(\Gamma=\bigcup_{j=1}^{k}(\Lambda_j+S)\), and \(\Gamma_0=\Gamma\setminus\bigcup_{j=1}^{k}\Lambda_j\). We fix \(n_1,\ldots,n_k\in\mathbb{N}\).
Let \(A\) denote the event \(\{M_{\Lambda_1}=n_1,\ldots,\ M_{\Lambda_k}=n_k,\ M_{\Gamma_0}=0\}\in\mathcal{F}_{\Gamma}\). Since for all \(\gamma\in\mathcal{C}\) and all \(\eta\in A\), we have \(U_{\Gamma}(\eta_{\Gamma}+\gamma_{\Gamma^c})=U(\eta_{\Gamma})\), from the DLR equation, we obtain
\begin{equation}
\begin{split}
P\big(M_{\Lambda_j}=n_j\, ; \, j=1,\ldots,k\big)&\geq P(A)\geq \int_{\mathcal{C}}1_{A}(\eta_{\Gamma})e^{-U(\eta_{\Gamma})}\pi^1(d\eta)\\
&\geq \prod_{j=1}^{k}\frac{1}{n_j!}|\Lambda_j|^{n_j}e^{-(\varphi(0)+\veps)n_j^2/2}e^{-|\Gamma|},
\end{split}
\end{equation}
where we use the fact that \(Z_{\Gamma}(\gamma)\leq1\) (see \eqref{partition upper}).

This implies that there exists \(n_0\in\mathbb{N}\) such that for any \(n_1,\ldots,n_k\geq n_0\),
\begin{equation}\label{lower prob}
\log P\big(M_{\Lambda_j}=n_j\, ; \, j=1,\ldots,k\big)\geq -(1+\veps)\frac{\varphi(0)+\veps}{2}\sum_{j=1}^{k}n_j^2.
\end{equation}
We note that
\begin{equation}
\sup_{|x|<\delta}V_{\eta}(x)\leq -\sum_{j=1}^{k}(u(x_j)-\veps)M_{\Lambda_j}(\eta),
\end{equation}
and from \eqref{lower} and \eqref{lower prob}, for all sufficiently small \(\lambda<0\), we obtain
\begin{equation}\label{pairwise leading}
\begin{split}
\log N(\lambda)&\geq\log P\big(-\sum_{j=1}^{k}(u(x_j)-\veps)M_{\Lambda_j}\leq(1+\veps)\lambda\big)\\
&\geq\log P\Big(M_{\Lambda_j}= \Big\lceil\frac{(1+\veps)c_j|\lambda|}{u(x_j)-\veps}\Big\rceil\,;\, j=1,\ldots,k\Big)\\
&\geq-(1+\veps)^4 \frac{\varphi(0)+\veps}{2}\sum_{j=1}^{k}\big(\frac{c_j}{u(x_j)-\veps}\big)^2\lambda^2,
\end{split}
\end{equation}
where \(c_j=u(x_j)^2/\sum_{j=1}^{k}u(x_j)^2\). 

From \eqref{epsilon0} and \eqref{pairwise leading}, we obtain \eqref{strong pair a}.
\end{proof}


\subsection{Proof of Theorem \ref{strong pair} (b)}

For the proof, we set the simple functions that approximate \(u\), where \(u(x)=-u_0(-x)\).
For every \(n\in\mathbb{N}\), we put
\begin{equation}\label{u_n}
u_{n}(x)=\sum_{\bm{j}\in\mathbb{Z}^d}u_{n,\bm{j}}1_{\Lambda_{n,\bm{j}}}(x),
\end{equation}
where \(u_{n,\bm{j}}=\sup_{y\in\Lambda_{n,\bm{j}}}u(y)\), and \(\Lambda_{n,\bm{j}}\) denotes the box \(\{x+\bm{j}/n\mid x\in(0,1/n]^d\}\). 

\begin{lemma}\label{upper2}
Let \(S\in\mathcal{B}(\mathbb{R}^d)\) satisfy condition \textbf{(S)}. Then for any \(b>0\), it holds that
\begin{equation}
\lim_{n\to\infty}\|u_n\|^2_{S_{b/n}}=\|u\|^2_{S},
\end{equation}
where \(S_{b/n}\) denotes the set \((S^c+B(0,b/n))^c\).
\end{lemma}

\begin{proof}
We note that \(u\) is a nonnegative continuous function with compact support.
Fix a constant \(b>0\). 
It is easy to see that for all sufficiently large \(n\in\mathbb{N}\),
\begin{equation}\label{upper2 eq1}
\|u_n\|_{S_{b/n}}^2\geq\|u\|_{S}^2.
\end{equation}
Since \(\mathrm{dist}(S^c,\alpha S)>0\) for any \(\alpha\in(0,1)\), we can find a monotonically non-increasing sequence \((\beta_n)_{n=n_0}^{\infty}\) \((n_0\geq 1)\) convergent to one, such that
\begin{equation}
\mathrm{dist}(S^c,\frac{1}{\beta_n}S)>b/n\quad(n\geq n_0),
\end{equation}
where for every \(\Lambda,\Gamma\in\mathcal{B}(\mathbb{R}^d)\), \(\mathrm{dist}(\Lambda,\Gamma)\) means \(\inf\{|x-y|\mid x\in\Lambda,\ y\in\Gamma\}\).
This implies that
\begin{equation}\label{S eq}
S\subset \beta_n S_{b/n}\quad(n\geq n_0).
\end{equation}
We fix \(\veps>0\). We find \(L>0\) such that \(\mathrm{supp}\,u_n\subset B(0,L)\) for all \(n\in\mathbb{N}\).
For each \(n\geq n_0\), we note that \(\|u_n\|_{S_{b/n}}<\infty\), and choose finite points \(x_{n,1},\ldots x_{n,k_n}\in\mathbb{R}^d\) such that \(|x_{n,j}|\leq L\), \(x_{n,i}-x_{n,j}\in (S_{b/n})^c\ (i\neq j)\), and
\begin{equation}\label{u_n epsilon}
\|u_n\|^2_{S_{b/n}}-\veps\leq\sum_{j=1}^{k_n}u_n(x_{n,j})^2.
\end{equation}
We can find a constant \(K>0\) independent of \(\veps\) such that \(k_n\leq K\) for all \(n\geq n_0\).
We put \(y_{n,j}=\beta_{n} x_{n,j}\). From \eqref{S eq}, we get 
\begin{equation}\label{x-y}
y_{n,i}-y_{n,j}\in \beta_n (S_{b/n})^c\subset S^c\quad(i\neq j),\quad |x_{n,j}-y_{n,j}|\leq (\beta_n -1)L.
\end{equation}
Since \(u\) is uniformly continuous, and \(u_n\) is uniformly convergent to \(u\), there exists \(\delta>0\) such that for all sufficiently large \(n\in\mathbb{N}\), \(|u_n(x)-u(y)|<\veps\) whenever \(|x-y|<\delta\).
Hence, from \eqref{u_n epsilon} and \eqref{x-y}, for all sufficiently large \(n\in\mathbb{N}\), we obtain
\begin{equation}\label{upper2 eq2}
\begin{split}
\|u_n\|_{S_{b/n}}^2-\veps &\leq\sum_{j=1}^{k_n}u(y_{n,j})^2+2\veps mK+\veps^2 K\\
&\leq\|u\|_{S}^2 +2\veps mK+\veps^2 K,
\end{split}
\end{equation}
where \(m\) is the maximum value of \(u\).
The lemma is proved by combining this and \eqref{upper2 eq1}.
\end{proof}
 
Let \(P\) be the distribution of the Gibbs point process in Theorem \ref{strong pair} (b).
The key to the proof of Theorem \ref{strong pair} (b) is the next proposition: the upper estimate of the Laplace functional as \(t\to\infty\).

\begin{proposition}\label{upper lap}
We put
\begin{equation}
v(x)=\sum_{j=1}^{k}v_j 1_{\Lambda_j}(x),
\end{equation}
where \(v_1,\ldots,v_k>0,\ \Lambda_1,\ldots,\Lambda_k\in\mathcal{B}(\mathbb{R}^d)\) are disjoint bounded sets with positive Lebesgue measure such that for each \(j=1,\ldots,k\), \(x-y\in S\) whenever \(x,y\in\Lambda_j\).
Then we have
\begin{equation}\label{upper lap eq}
\limsup_{t\to\infty}t^{-2} \log\big(\int_{\mathcal{C}}\exp\big(t\int_{\mathbb{R}^d}v(x)\eta(dx)\big)P(d\eta)\big)\leq \frac{1}{2a}\max_{J\in K}\sum_{j\in J}{v_j}^2,
\end{equation}
where \(K\) denotes the set
\begin{equation}
\begin{split}
\{J\subset\{1,\ldots,k\}\mid &\text{for all \(i\neq j\in J\),}\\
&\quad\text{there are \(x\in\Lambda_{i}\) and \(y\in\Lambda_{j}\) such that \(x-y\in S^c\)}\}.
\end{split}
\end{equation}
\end{proposition} 

The proof of Proposition \ref{upper lap} proceeds via two lemmas.

\begin{lemma}\label{upper prob}
For the sets \(\Lambda_1,\ldots,\Lambda_k\in\mathcal{B}(\mathbb{R}^d)\) in Proposition \ref{upper lap} and any \(\veps\in(0,1)\), there exists \(n_0\in\mathbb{N}\) such that for any \(n_1,\ldots,n_k\geq n_0\),
\begin{equation}
\log P(M_{\Lambda_j}=n_j\,;\,j=1,\ldots,k)\leq -(1-\veps)\frac{a}{2}\sum_{j=1}^{k}n_j^2-a\sum_{(i,j)\in I}n_i n_j,
\end{equation}
where \(I=\{(i,j)\in\mathbb{N}^2\mid \text{\(1\leq i<j\leq k\), and \(x-y\in S\) for any \(x\in\Lambda_i,\ y\in\Lambda_j\)}\}\).
\end{lemma}

\begin{proof}
Let \(A\) denote the event \(\{M_{\Lambda_j}=n_j\,;\,j=1,\ldots,k\}\) and \(\Lambda=\bigcup_{j=1}^{k}\Lambda_j\).
It is obvious that \(U_{\Lambda}(\eta_{\Lambda}+\gamma_{\Lambda^c})\geq U(\eta_{\Lambda})\) for all \(\eta,\gamma\in\mathcal{C}\).
By the DLR equation and \eqref{partition lower}, we obtain
\begin{equation}
\begin{split}
P(A)&\leq e^{|\Lambda|}\int_{\mathcal{C}}1_{A}(\eta_{\Lambda})e^{-U(\eta_{\Lambda})}\pi^1(d\eta)\\
&\leq \prod_{j=1}^{k}\frac{|\Lambda_j|^{n_j}}{n_j!}\exp\big(-\frac{a}{2}\sum_{j=1}^{k}n_j(n_j-1)-a\sum_{(i,j)\in I}n_i n_j\big),
\end{split}
\end{equation}
which proves the lemma.
\end{proof}

\begin{lemma}\label{int lem}
Let \(c>0,\ v_1,\ldots,v_k>0\), and \(I\) be a subset of \(\{(i,j)\in\mathbb{N}^2\mid 1\leq i<j\leq k\}\).
Then, for any \(\veps>0\), there exists \(T>0\) such that for all \(t>T\),
\begin{equation}\label{int eq}
\begin{split}
&\sum_{n_1=0}^{\infty}\cdots\sum_{n_k=0}^{\infty} \exp\big(-c\sum_{j=1}^{k}n_j^2-2c\sum_{(i,j)\in I}n_i n_j+t\sum_{j=1}^{k}v_j n_j\big)\\
&\leq \exp\big((1+\veps)(\max_{J\in K_I}\sum_{j\in J}v_j^2)\frac{t^2}{4c}\big),\\
\end{split}
\end{equation} 
where \(K_I=\{J\subset\{1,\ldots,k\}\mid\text{\((i,j)\notin I\) whenever \(i,j\in J\)}\}\). 
\end{lemma}

The idea of the proof is to change of variables:
\begin{equation}
\sum_{n_1=0}^{\infty}\, \sum_{n_2=0}^{\infty} \exp(-(n_1+n_2)^2)= \sum_{m_1=0}^{\infty} \ \sum_{m_2\in\{m_1,m_1-2,\ldots,-m_1\}}\exp(-m_1^2)
\end{equation}
where we put \(m_1=n_1+n_2,\ m_2=n_1-n_2\).

\begin{proof}[Proof of Lemma \ref{int lem}]
We fix \(\veps>0\).
Let \(G(t;I)\) denote the left-hand side of \eqref{int eq}.
We have
\begin{equation}\label{int empty}
\begin{split}
G(t;\emptyset)&=\sum_{n_1=0}^{\infty}\cdots\sum_{n_k=0}^{\infty} \exp\big(-c\sum_{j=1}^{k}n_j^2+t\sum_{j=1}^{k}v_j n_j\big)\\
&\leq \prod_{j=1}^{k}\big(\int_0^{\infty}\exp(-cx^2+tv_jx)dx+\max_{x\in[0,\infty)}\exp(-cx^2+tv_jx)\big)\\
&\leq \exp\big((1+\veps)\sum_{j=1}^{k}v_j^2\frac{t^2}{4c}\big),\\ 
\end{split}
\end{equation}
for all sufficiently large \(t>0\).

Now we investigate the case that \(I\neq\emptyset\). We assume that \((1,2)\in I\) without loss of generality.
We consider
\begin{equation}\label{int left}
G_{0}(t)=\sum_{n_1=0}^{\infty}\,\sum_{n_2=0}^{\infty}\exp\big(-c(n_1^2+n_2^2)-2c(n_1 n_2+ n_1 X+n_2 Y)+t(v_1 n_1+v_2 n_2)\big)
\end{equation}
in the left-hand side of \eqref{int eq}, where we put
\begin{equation}
X=\sum_{\substack{(1,j)\in I\\ j\neq 2}}n_j,\quad Y=\sum_{(2,j)\in I}n_j.
\end{equation}
We put \(m_1=n_1+n_2,\ m_2=n_1-n_2\) and have
\begin{equation}
\begin{split}
G_{0}(t) &= \sum_{m_1=0}^{\infty}\exp\big(-cm_1^2-(cX+cY-\frac{tv_1}{2}-\frac{tv_2}{2})m_1\big)\\
&\quad\times\sum_{m_2\in\{m_1,m_1-2,\ldots,-m_1\}}\exp\big((-cX+cY+\frac{tv_1}{2}-\frac{tv_2}{2})m_2\big)\\
&\leq \sum_{m=0}^{\infty}(m+1)\exp\big(-cm^2-(cX+cY-\frac{tv_1}{2}-\frac{tv_2}{2})m\big)\\
&\quad \times \bigg\{\exp\big((-cX+cY+\frac{tv_1}{2}-\frac{tv_2}{2})m\big)\\
&\quad\quad\quad\quad\quad\quad\quad\quad\quad\quad\quad +\exp\big((cX-cY-\frac{tv_1}{2}+\frac{tv_2}{2})m\big)\bigg\}\\
&\leq \sum_{n_1=0}^{\infty}\exp(-cn_1^2-2cn_1X+(1+\veps)tv_1n_1)\\
&\quad\quad\quad\quad\quad\quad\quad\quad+ \sum_{n_2=0}^{\infty}\exp(-cn_2^2-2cn_2Y+(1+\veps)tv_2n_2),\\
\end{split}
\end{equation} 
for all \(t>T'\), where \(T'=1/(\veps\min_j v_j)\).
Hence, for all \(t>T'\), we obtain
\begin{equation}
\begin{split}
&G(t;I)\\
&\leq \sum_{n_j\geq0,\ \text{for}\ j\neq2} \exp\big(-c\sum_{j\neq 2}n_j^2-2c\sum_{(i,j)\in I_2}n_in_j+(1+\veps)t\sum_{j\neq2}v_jn_j\big) \\
&\quad+\sum_{n_j\geq0\ \text{for}\ j\neq1} \exp\big(-c\sum_{j\neq 1}n_j^2-2c\sum_{(i,j)\in I_1}n_in_j+(1+\veps)t\sum_{j\neq1}v_jn_j\big),
\end{split}
\end{equation}
where \(I_l=\{(i,j)\in I\mid i,j\neq l\}\ (l=1,2)\).

If \(I_1\neq\emptyset\) or \(I_2\neq\emptyset\), the above procedure can be repeated.
We iterate \(2^{k-1}-1\) times at most, and for all \(t>T'\), we obtain
\begin{equation}\label{int reduce}
\begin{split}
G(t;I)\leq &2^k\!\sum_{\{m_1,\ldots,m_l\}\in K_I}\\
&\times \sum_{n_1=0}^{\infty}\cdots\sum_{n_l=0}^{\infty}\exp\big(-c\sum_{j=1}^{l}n_j^2+(1+\veps)^k t\sum_{j=1}^{l}  v_{m_j}n_j\big).
\end{split}
\end{equation}
The lemma is proved by \eqref{int empty} and \eqref{int reduce}. 
\end{proof}

\begin{proof}[Proof of Proposition \ref{upper lap}]
We fix \(\veps\in(0,1)\). Let \([k]\) denote the set \(\{1,\ldots,k\}\).

From Lemma \ref{upper prob}, for \(n_0\in\mathbb{N}\) large enough, we have
\begin{equation}
\begin{split}
&\int_{\mathcal{C}}\exp\big(t\int_{\mathbb{R}^d}v(x)\eta(dx)\big)P(d\eta)\\
&=\sum_{J\subset [k]}\ \sum_{\substack{n_j>n_0\ \text{for}\ j\in J, \\ 0\leq n_j\leq n_0\ \text{for}\ j\in [k]\setminus J}}
\exp\big(t\sum_{j=1}^{k}v_j n_j\big)P(M_{\Lambda_j}=n_j\,;\,j\in [k])\\
&\leq k(n_0+1)\exp\big(tn_0\sum_{j=1}^{k}v_{j}\big)+
\sum_{J(\neq\emptyset)\subset [k]} k(n_0+1)\exp\big(tn_0\sum_{j\in [k]\setminus J}v_{j}\big)\\
&\quad\qquad\qquad\qquad\times \sum_{n_j>n_0\ \text{for}\ j\in J} \exp\big(t\sum_{j\in J} v_{j} n_{j}\big)P(M_{\Lambda_{j}}=n_{j}\,;\,j\in J)\\
&\leq \,k(n_0+1)\exp\big(tn_0\sum_{j=1}^{k}v_j\big)\bigg\{1+\sum_{J(\neq\emptyset)\subset [k]}\ \ \sum_{n_j\geq0\ \text{for}\ j\in J}\\
&\qquad\times\exp\big(-(1-\veps)\frac{a}{2}\sum_{j\in J}n_j^2-(1-\veps)a\sum_{(i,j)\in I(J)}n_i n_j +t\sum_{j\in J}v_j n_j \big) \bigg\},
\end{split}
\end{equation}
where \(I(J)=\{(i,j)\in J\times J \mid \text{\(i<j\), and \(x-y\in S\) for any \(x\in\Lambda_i,\ y\in\Lambda_j\)}\}\).

From Lemma \ref{int lem}, for any nonempty subset \(J\subset\{1,\ldots,k\}\) and all sufficiently large \(t>0\), we obtain
\begin{equation}
\begin{split}
&\sum_{n_j\geq0\ \text{for}\ j\in J}\exp\big(-(1-\veps)\frac{a}{2}\sum_{j\in J}n_j^2-(1-\veps)a\sum_{(i,j)\in I(J)}n_i n_j +t\sum_{j\in J}v_j n_j \big)\\
&\leq \exp\big(\frac{1+\veps}{1-\veps}(\max_{J'\in K(J)}\sum_{j\in J'}{v_j}^2) \frac{t^2}{2a}\big),
\end{split}
\end{equation}
where \(K(J)=\{J'\subset J\mid\text{\((i,j)\neq I(J)\) whenever \(i,j\in J'\)}\}\).

Since \(K(J)\subset K\), we obtain \eqref{upper lap eq}.
\end{proof}

\begin{proof}[Proof of Theorem \ref{strong pair} (b)]
We put \(u(x)=-u_{0}(-x)\) and \(u_n\) as \eqref{u_n}. We note that for each \(n\in\mathbb{N}\), \(u_{n,\bm{j}}=0\) except for finitely many \(\bm{j}\in\mathbb{Z}^d\). We fix \(\veps>0\).
From Proposition \ref{upper lap}, for all sufficiently large \(n\in\mathbb{N}\) and all sufficiently large \(t>0\), we have
\begin{equation}
\begin{split}
&\log\big(\int_{\mathcal{C}}\exp\big(t\int_{\mathbb{R}^d} u(x)\eta(dx)\big)P(d\eta)\big)\\
&\leq \log\big(\int_{\mathcal{C}}\exp\big(t\int_{\mathbb{R}^d} u_n(x) \eta(dx)\big)P(d\eta)\big)\\ 
&\leq (1+\veps)(\max_{J\in K_n}\sum_{\bm{j}\in J}{u_{n,\bm{j}}}^2)\frac{t^2}{2a},\\
\end{split}
\end{equation}
where for each \(n\in\mathbb{N}\), we put
\begin{multline}
K_n=\{J\subset \mathbb{Z}^d \mid \text{ for all \(\bm{i}\neq\bm{j}\in J\), there exist \(x\in\Lambda_{n,\bm{i}}\) and \(y\in\Lambda_{n,\bm{j}}\)}\\
\text{such that \(x-y\in S^c\) } \}.
\end{multline}
If \(\{\bm{i},\bm{j}\}\in K_n\ (\bm{i}\neq \bm{j})\), then for any \(x\in\Lambda_{n,\bm{i}}\) and \(y\in\Lambda_{n,\bm{j}}\), we have \(x-y\in (S_{2\sqrt{d}/n})^c\).
This implies that
\begin{equation}\label{lap S}
\log\big(\int_{\mathcal{C}}\exp\big(t\int_{\mathbb{R}^d} u(x)\eta(dx)\big)P(d\eta)\big)\leq (1+\veps)\|u_n\|^2_{S_{2\sqrt{d}/n}}\frac{t^2}{2a},
\end{equation}
for all sufficiently large \(n\in\mathbb{N}\) and all sufficiently large \(t>0\).

We put
\begin{equation}
t=\frac{a|\lambda|}{(1+\veps)\|u_n\|^2_{S_{2\sqrt{d}/n}}},
\end{equation}
and from \eqref{upper}, \eqref{lap S}, and Lemma \ref{upper2} we obtain \eqref{strong pair b}.
\end{proof}

\section*{Acknowledgments}
The author would like to thank Professor Naomasa Ueki for the useful discussions.
The author is grateful to Kumano Dormitory, Kyoto University for their generous financial and living assistance.

\end{document}